\documentclass{fundam}
\usepackage{amsmath,amssymb,amsfonts}
\usepackage[english]{babel}
\usepackage{graphicx}
\usepackage{gastex}
\usepackage{cite}

\def\Pal{{\sf Pal}}
\def\pal{{\cal P}}
\def\E{{\sf E}}

\begin{document}

\title{The Number of Distinct Subpalindromes in Random Words}
\author{Mikhail Rubinchik\\ 
Ural Federal University\\
Ekaterinburg, Russia\\
{\tt mikhail.rubinchik@gmail.com}
\and Arseny M. Shur\corresponding \\ 
Ural Federal University\\
Ekaterinburg, Russia\\
{\tt arseny.shur@urfu.ru}
}
\runninghead{M. Rubinchik and A.M. Shur}{The Number of Distinct Subpalindromes in Random Words}
\maketitle

\vspace*{-2.5cm}
\begin{abstract}
We prove that a random word of length $n$ over a $k$-ary fixed alphabet contains, on expectation, $\Theta(\sqrt{n})$ distinct palindromic factors. We study this number of factors, $\E(n,k)$, in detail, showing that the limit $\lim_{n\to\infty}\E(n,k)/\sqrt{n}$ does not exist for any $k\ge2$, $\liminf_{n\to\infty}\E(n,k)/\sqrt{n}=\Theta(1)$, and $\limsup_{n\to\infty}\E(n,k)/\sqrt{n}=\Theta(\sqrt{k})$. Such a complicated behaviour stems from the asymmetry between the palindromes of even and odd length. We show that a similar, but much simpler, result on the expected number of squares in random words holds. We also provide some experimental data on the number of palindromic factors in random words.
\end{abstract}

\section{Introduction}

Palindromes are among the most important and actively studied repetitions in words. Recall that a word $w=a_1\cdots a_n$ is a palindrome if $a_1\cdots a_n=a_n\cdots a_1$. In particular, all letters are palindromes; the empty word is also considered as a palindrome, but throughout this paper we do not count it. Palindromes are objects of intensive study since 1970s. One direction of this study is formed by different counting problems; see, for example, \cite{RiSh14}, where the asymptotic growth of the language of \emph{palstars} (words that are concatenations of even-length palindromes) is found. An important group of problems within this direction concerns the possible number of distinct palindromic factors, or subpalindromes, in a word. We call this number \emph{palindromic richness}. 

Clearly, for the words containing $k$ different letters the lower bound for their palindromic richness is $k$. If $k>2$, then this bound is sharp, since the infinite periodic word $(a_1\cdots a_k)^\omega$, where $a_1,\ldots, a_k$ are different letters, has no subpalindromes except letters. For $k=2$ the lower bound is less straightforward: the minimum richness of an infinite word is 8 and the minimum richness of an \emph{aperiodic} infinite word is 10 \cite{FiZa13}. (Moreover, the minimum richness of a finite word of length $\ge 9$ is 8.) On the other hand, the maximum richness of an $n$-letter word over any alphabet is $n$, as was first observed in \cite{DJP01}. Such ``rich'' words are objects of intensive study (see, e.g., \cite{GJWZ09}). Still, little is known about the number of rich words of a given length. Currently, the best lower bound on the number of binary rich words is of the form $\frac{C^{\sqrt{n}}}{p(n)}$, where $p(n)$ is a polynomial and $C\approx 37$ \cite{GSS15}. In the same paper, it was conjectured that this number is upper bounded by $n^{\sqrt{n}}$, while the best proved upper bound is of order $1.605^n$. Anyway, most of the words are not rich, and it is quite interesting to see how the palindromic richness behaves in the generic case. We will show, in a straightforward way, that any richness between the two extremums is reachable:

\begin{proposition} \label{ktoN}
Any number between 8 and $n$ in the binary case, and between $k$ and $n$ in the $k$-ary case with $k>2$ is the palindromic richness of some word of length $n$.
\end{proposition}

So, the following question is quite natural:
\begin{center}
\emph{what is the expected palindromic richness of a random word of length $n$?}
\end{center}
The following theorem, which is our main result, provides a detailed answer to this question. Note that the bigger is the alphabet, the less probable is that a random word will be a palindrome; so, statements 3 and 4 of this theorem seem rather unexpected.

\begin{theorem} \label{main}
Let $k\ge 2$.\\
(1) The expected palindromic richness $\E(n,k)$ of a random $k$-ary word of length $n$ is $\Theta(\sqrt{n})$ as $n\to \infty$ with $k$ fixed.\\
(2) The ratio $\frac{\E(n,k)}{\sqrt{n}}$ has no limit as $n\to \infty$ with $k$ fixed.\\
(3) The function $\underline{C}(k)=\liminf_{n\to\infty} \frac{\E(n,k)}{\sqrt{n}}$ is $\Theta(1)$ as $k\to \infty$.\\
(4) The function $\overline{C}(k)=\limsup_{n\to\infty} \frac{\E(n,k)}{\sqrt{n}}$ is $\Theta(\sqrt{k})$ as $k\to \infty$.
\end{theorem}

We also give more precise theoretical estimation of the quantities $\underline{C}(k)$ and $\overline{C}(k)$ for some alphabets and compare them to the results of our experiments. Finally, we show that our technique allows one to get, in a much easier way, the bound $\Theta(\sqrt{n})$ on the number of \emph{squares} in a random word. 

The text is organized as follows. Section 2 contains notation, definitions, and the proof of Proposition~\ref{ktoN}. In Sections 3--5 we prove Theorem~\ref{main}. In Sect.~3, we prove the upper bound $O(\sqrt{n})$ and find the range of lengths, containing the main part of all distinct palindromic factors. Then in Sect.~4--5 we study the probability of getting a palindromic factor of a given length from a prescribed range, 
using the results of Guibas and Odlyzko \cite{GuOd78,GuOd81} on factor avoidance. The final Sect.~6 is devoted to numerical studies and to extending  our methods to counting the expected number of squares instead of palindromes.

\section{Preliminaries}

We study non-empty words over finite alphabets, using the array notation $w=w[1..n]$ when appropriate and writing $|w|$ for the length of $w$. Any word $w[i..j]$, where $1\le i\le j\le n$, is a \emph{factor} of $w$; a factor of the form $w[1..j]$ (resp., $w[i..n]$) is called a \emph{prefix} (resp., a \emph{suffix}) of $w$. A \emph{square} is any word of the form $ww$. By $u^\omega$ we denote the right-infinite word obtained by concatenation of an infinite sequence of copies of the word $u$.

A word satisfying $w[i]=w[n{-}i]$ for all $i=1,\ldots,n$, is a palindrome. \emph{Palindromic richness} of a word $w$ is the number of distinct palindromes which are factors of $w$.

By a \emph{random $k$-ary word of length $n$} we mean the random variable equidistributed among all $k$-ary words of length $n$. The \emph{expected} palindromic richness $\E(n,k)$ of this random word is the main characteristic studied in this paper.

Throughout the paper, the notation $\log$ always stands for the base $k$ logarithm; the natural logarithm is denoted by $\ln$.

\medskip
\textbf{Proof of Proposition~\ref{ktoN}:}\\[2pt]
Let $k>2$ and $w=(a_1\cdots a_k)^\omega$. The word $a_1^{l-k}w[1..n{-}l{+}k]$ of length $n$ has exactly $l$ palindromes: all letters plus the palindromes $a_1^i$ for $i=2,\ldots,l{-}k{+}1$. Since $l$ can be an arbitrary integer between $k$ and $n$, we are done with this case.

Now consider the binary alphabet $\{0,1\}$. The infinite word $u=(001101)^\omega$ has exactly 8 palindromic factors: $0,1,00,11,010,101,0110,1001$. All of them appear in $u[1..9]$. Then the word $0^{l-8}u[1..n{-}l{+}8]$ of length $n$ has exactly $l$ palindromes for any $l=k,\ldots,n-1$: those of $u$ plus $0^3,\ldots,0^{l-6}$. Since the words of length $n$ and richness $n$ exist (for example, $0^n$), we get the desired result.

\section{A simple upper bound}

The aim of this section is to prove that $\E(n,k)=O(\sqrt{n})$ for any fixed $k$ and to show that the most part of palindromic factors in a word of length $n$ has the length close to $\log n$. The first two lemmas are straightforward.

\begin{lemma} \label{pal_m}
The number of distinct $k$-ary palindromes of length $m$ is $\Pal(k,m)=k^{\lceil m/2\rceil}$.
\end{lemma}

\begin{proof}
The mentioned quantity is the number of ways to choose the first $\lceil m/2\rceil$ letters of a word of length $m$. If this word is a palindrome, the remaining letters are determined uniquely.
\end{proof}

\begin{lemma} \label{pal_E}
The expected number of palindromic factors\footnote{Not necessarily distinct!} of length $m$ in a $k$-ary word of length $n$ is $\hat \E(n,k,m)=\frac{n-m+1}{k^{\lfloor m/2\rfloor}}$.
\end{lemma}

\begin{proof}
The probability for a $k$-ary word of length $m$ to be a palindrome is $\frac{k^{\lceil m/2\rceil}}{k^m}=\frac{1}{k^{\lfloor m/2\rfloor}}$ by Lemma~\ref{pal_m}. This probability obviously coincides with the expected number of palindromic factors of length $m$ in the fixed position of a word of length $n$. Now the lemma follows by the linearity of expectation, because a word of length $n$ has $n{-}m{+}1$ factors of length $m$.
\end{proof}

The following combinatorial lemma is used in the proof of Lemma~\ref{upper}.

\begin{lemma} \label{summ}
$\sum_{i=c}^\infty\frac{i+1}{k^i}=\frac{(c+1)k-c}{k^{c-1}(k-1)^2}$.
\end{lemma}

\begin{proof}
The following sequence of transformations holds:
\begin{multline*}
\sum_{i=c}^\infty(i+1)x^i=\left(\sum_{i=c+1}^\infty x^i\right)'=\left(\frac{x^{c+1}}{1-x}\right)'\\
=\frac{(c+1)x^c(1-x)+x^{c+1}}{(1-x)^2}=\frac{(c+1)x^c-cx^{c+1}}{(1-x)^2}=\left[x=1/k\right]=\frac{(c+1)k-c}{k^{c-1}(k-1)^2}.
\end{multline*}
\end{proof}

In the rest of this section we prove the following upper bound on the expected palindromic richness. Some notions and formulas from the proof will be then used throughout the rest of the paper.

\begin{lemma} \label{upper}
For any fixed $k\ge 2$ one has $\E(n,k)\le\sqrt{n}(\sqrt{k}+O(1))$.
\end{lemma}

\begin{proof}
Let $w$ be a word picked up uniformly at random from the set of all $k$-ary words of length $n$. It is clear that the expected number $\E(w,m)$ of distinct palindromic factors of length $m$ in $w$ can exceed neither $\Pal(k,m)$ nor $\hat \E(n,k,m)$. So we have the following upper bound:
\begin{equation} \label{upp}
\E(n,k)\le\sum_{m=0}^n\min\{\Pal(k,m),\hat \E(n,k,m)\}.
\end{equation}
Since the formulas given in Lemmas~\ref{pal_m} and~\ref{pal_E} are asymmetric with respect to the parity of $m$, it is convenient to split the sum in \eqref{upp} into two sums, corresponding to even and odd values of $m$, respectively, and compute them separately. So we have
\begin{subequations}
\begin{align}
\Pal(k,2m)=k^m,\quad&\hat \E(n,k,2m)=\frac{n-2m+1}{k^m},\label{upp_even}\\
\Pal(k,2m{+}1)=k^{m+1},\quad&\hat \E(n,k,2m{+}1)=\frac{n-2m}{k^m},\label{upp_odd}
\end{align}
\end{subequations}
and then we can write
\begin{multline} \label{upp2}
\E(n,k)=\E_e(n,k)+\E_o(n,k)\\
\le \sum_{m=0}^{\lfloor n/2\rfloor}\min\{\Pal(k,2m),\hat \E(n,k,2m)\} + \!\!\!\!\sum_{m=0}^{\lfloor (n-1)/2\rfloor}\min\{\Pal(k,2m{+}1),\hat \E(n,k,2m{+}1)\}.
\end{multline}
The graphs of \eqref{upp_even} and \eqref{upp_odd} as functions of $m$ (for $k$ and $n$ fixed) are drawn in Fig.~\ref{graphs}. 

\begin{figure}[htb]
\centerline{ 
\unitlength=1mm
\begin{picture}(66,60)(0,1)
\gasset{Nw=0,Nh=0,AHangle=15,AHLength=2.5,ExtNL=y,NLdist=1}
\node(o)(0,0){}
\node[NLangle=60](x)(55,0){${m}$}
\node[NLangle=340](y)(0,60){}
\node[Nw=1,Nh=1,Nfill=y](p)(25,34.8){}
\node[NLangle=45](p1)(25,0){$p_e$}
\drawedge(o,x){}
\drawedge(o,y){}
\drawedge[AHnb=0,dash={1 1}{0}](p,p1){}
\put(30,55){\makebox(0,0)[lt]{$k^m$}}
\put(37,8){\makebox(0,0)[lb]{$\frac{n-2m+1}{k^m}$}}
\drawcbezier[AHnb=0,linewidth=0.2](0,2,15,2.5,26,28,30,60)
\drawcbezier[AHnb=0,linewidth=0.2](20,60,24,28,35,1.5,50,0)
\end{picture} 
 \begin{picture}(57,60)(0,1)
\gasset{Nw=0,Nh=0,AHangle=15,AHLength=2.5,ExtNL=y,NLdist=1}
\node(o)(0,0){}
\node[NLangle=60](x)(55,0){${m}$}
\node[NLangle=340](y)(0,60){}
\node[Nw=1,Nh=1,Nfill=y](p)(22.5,44.8){}
\node[NLangle=45](p1)(22.5,0){$p_o$}
\drawedge(o,x){}
\drawedge(o,y){}
\drawedge[AHnb=0,dash={1 1}{0}](p,p1){}
\put(25,55){\makebox(0,0)[lt]{$k^{m+1}$}}
\put(37,8){\makebox(0,0)[lb]{$\frac{n-2m}{k^m}$}}
\drawcbezier[AHnb=0,linewidth=0.2](0,3.5,10,5,21,27,25,60)
\drawcbezier[AHnb=0,linewidth=0.2](20,60,24,28,35,1.5,50,0)
\end{picture} 
}
\caption{The graphs of $\Pal$ and $\hat\E$ for even-length (left) and odd-length (right) palindromes.} \label{graphs}
\vspace*{-2mm}
\end{figure}
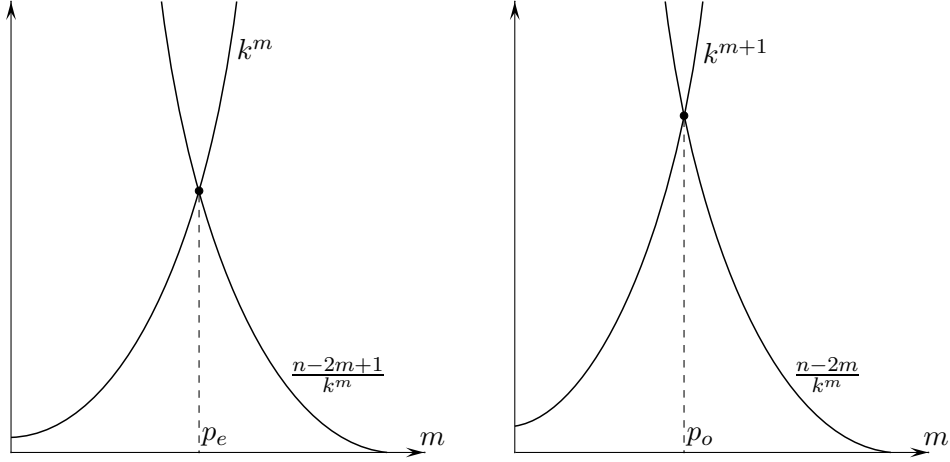
So, in each case we have to find the point of intersection of two graphs and then sum up all values of $\Pal$ to the left of this point and all values of $\hat\E$ to the right of this point. We start with even-length palindromes. Recall that $\log$ denotes the base $k$ logarithm.

The intersection point $p_e=p_e(n,k)$ is given by the equation $k^{2m}=n-2m+1$, so $p_e\approx\frac{\log n}2$. Using standard transformations and the Maclaurin series for $\ln(1-x)$, we get a more precise estimate:
\begin{multline} \label{pe}
p_e=\frac{\log(n-2p_e+1)}2=\frac{\log(n-\log(n-2p_e+1)+1)}2\\
=\frac 12\cdot\Big(\log n+\log\Big(1-\frac{\log(n-2p_e+1)-1}{n}\Big)\Big)=\frac{\log n}2-\frac{\log(n-2p_e+1)-1}{(2\ln k)\cdot n}+O\Big(\frac{\log^2n}{n^2}\Big)\\
=\frac{\log n}2-\frac{\log n-1}{(2\ln k)\cdot n}+O\Big(\frac{\log^2n}{n^2}\Big)\,.
\end{multline}
Replacing geometric sequences by geometric series and applying Lemma~\ref{summ}, we obtain
\begin{multline} \label{even1}
\E_e(n,k)\le \sum_{m=0}^{\lfloor p_e\rfloor}k^m+\!\!\!\!\sum_{m=\lfloor p_e\rfloor{+}1}^{\lfloor n/2\rfloor}\!\!\!\!\frac{n-2m+1}{k^m}\le
\frac{k^{\lfloor p_e\rfloor}}{1-1/k}+\frac{n+1}{k^{\lfloor p_e\rfloor+1}(1-1/k)}- \frac2k\cdot\!\!\sum_{m=\lfloor p_e\rfloor}^{\infty}\frac{m+1}{k^m}\\
=\frac{k^{\lfloor p_e\rfloor+1}}{k-1}+\frac{n+1}{k^{\lfloor p_e\rfloor}(k-1)}-
\frac{2(\lfloor p_e\rfloor+1)k-2\lfloor p_e\rfloor}{k^{\lfloor p_e\rfloor}(k-1)^2}\,.
\end{multline}
Using \eqref{pe} and the Maclaurin series for the exponential function, we compute
\begin{equation} \label{kpe}
k^{\lfloor p_e\rfloor}=\frac{k^{p_e}}{k^{\{p_e\}}}= \frac{\sqrt{n}\cdot k^{-\frac{\log n-1}{(2\ln k)\cdot n}+ O(\frac{\log^2n}{n^2})}}{k^{\{p_e\}}}= \frac{\sqrt{n}\cdot\big(1-\frac{\log n-1}{2n}+ O(\frac{\log^2n}{n^2})\big)}{k^{\{p_e\}}}.
\end{equation}
Substituting \eqref{kpe} and \eqref{pe} into \eqref{even1}, we finally obtain
\begin{equation} \label{even2}
\E_e(n,k)\le\frac{\sqrt{n}\cdot k^{1-\{p_e\}}}{k-1}+\frac{\sqrt{n}\cdot k^{\{p_e\}}}{k-1}+O\Big(\frac{\log n}{\sqrt{n}}\Big)\,.
\end{equation}
Note that the constant inside the $O$-term can be chosen independent of $k$. Now we proceed with the odd-length palindromes. The following property of the intersection point $p_o=p_o(n,k)$ is quite useful.

\begin{lemma} \label{le:pepo}
$p_e=p_o+1/2$.
\end{lemma}

\begin{proof}
Recall that $p_o$ is the root of the equation $k^{2p_o+1}=n-2p_o$, so $p_o=\log\sqrt{n-2p_o}-1/2$. Similarly, $p_e=\log\sqrt{n-2p_e+1}$. Then
\begin{equation} \label{pepo}
p_e-p_o=\frac 12+\log\sqrt{\frac{n-2p_e+1}{n-2p_o}} \,.
\end{equation}
Denoting the logarithm in \eqref{pepo} by $\Delta$, we obtain
\begin{equation} \label{delta}
\Delta=\log\sqrt{\frac{n-2p_o-1-2\Delta+1}{n-2p_o}}=\log\sqrt{1-\frac{2\Delta}{n-2p_o}} \,.
\end{equation}
If $\Delta>0$, then the square root in \eqref{delta} is less than 1, implying $\Delta<0$. Similarly,  if $\Delta<0$, then the square root in \eqref{delta} is greater than 1, implying $\Delta>0$. These contradictions show that the only possible case is $\Delta=0$, whence the result.
\end{proof}

Lemma~\ref{le:pepo} and \eqref{pe} give us 
\begin{equation} \label{po}
p_o=\frac{\log n - 1}2-\frac{\log n-1}{(2\ln k)\cdot n}+O\Big(\frac{\log^2n}{n^2}\Big)\,.
\end{equation}
Similar to the even case we obtain
\begin{multline} \label{odd1}
\E_o(n,k)=\sum_{m=0}^{\lfloor p_o\rfloor}k^{m+1}+\!\!\!\!\sum_{m=\lfloor p_o\rfloor{+}1}^{\lfloor (n{-}1)/2\rfloor}\!\!\!\frac{n-2m}{k^m}\le
\frac{k^{\lfloor p_o\rfloor+1}}{1-1/k}+\frac{n}{k^{\lfloor p_o\rfloor+1}(1-1/k)}- \frac2k\cdot\!\!\sum_{m=\lfloor p_o\rfloor}^{\infty}\frac{m+1}{k^m}\\
=\frac{k^{\lfloor p_o\rfloor+2}}{k-1}+\frac{n}{k^{\lfloor p_o\rfloor}(k-1)}-
\frac{2(\lfloor p_o\rfloor+1)k-2\lfloor p_o\rfloor}{k^{\lfloor p_o\rfloor}(k-1)^2}\,.
\end{multline}
From \eqref{kpe} and Lemma~\ref{le:pepo} we have
\begin{equation} \label{kpo}
k^{\lfloor p_o\rfloor}=\frac{\sqrt{n}\cdot\big(1-\frac{\log n-1}{2n}+ O(\frac{\log^2n}{n^2})\big)}{k^{\{p_o\}+1/2}}\,.
\end{equation}
Substituting \eqref{kpo} and \eqref{po} into \eqref{odd1}, we finally get
\begin{equation} \label{odd2}
\E_o(n,k)\le\frac{\sqrt{n}\cdot k^{3/2-\{p_o\}}}{k-1}+\frac{\sqrt{n}\cdot k^{1/2+\{p_o\}}}{k-1}+O\Big(\frac{\log n\cdot\sqrt{k}}{\sqrt{n}}\Big)\,,
\end{equation}
and from \eqref{even2} and \eqref{odd2}
\begin{equation} \label{e1}
\E(n,k)\le \frac{\sqrt{n}\cdot \big(\sqrt{k}\cdot(k^{1-\{p_o\}}+k^{\{p_o\}}) +(k^{1-\{p_e\}}+k^{\{p_e\}})\big)}{k-1}+
O\Big(\frac{\log n\cdot\sqrt{k}}{\sqrt{n}}\Big)\,,
\end{equation}
whence the result.
\end{proof}

\begin{remark} \label{antiphase}
According to Lemma~\ref{le:pepo}, the expressions in internal parentheses in \eqref{e1} oscillate in antiphase. So, if $\{p_o\}\approx0$ (i.e., $n$ is slightly bigger than an odd power of $k$), the bound \eqref{e1} approaches its maximum and approximates to $\sqrt{n}(\sqrt{k}+\frac{4\sqrt{k}}{k-1})$, and if $\{p_e\}\approx0$ (i.e., $n$ is slightly bigger than an even power of $k$), this bound goes to minimum values close to $\sqrt{n}(3+\frac{4}{k-1})$. 
\end{remark}

The given upper bounds leave an impression that for any fixed $k$ the function $\E(n,k)$ oscillates between its low values close to $C\sqrt{n}$ for some absolute constant $C$ and its high values close to $D\sqrt{nk}$ for some absolute constant $D$. But the bound \eqref{e1} is somewhat imprecise, because the initial bound \eqref{upp} is generous enough. Indeed, if the number of palindromic factors of length $m$ in a word is greater than the number of distinct palindromes of this length, still some palindromes of length $m$ can be missing from this word. Similarly, if the number of these factors of length $m$ in a word is less than the number of distinct palindromes of this length, some of the factors can repeat, decreasing the number of distinct palindromes. Since the probability of an event ``to contain a given palindrome of length $m$'' depends not only on $n,k$, and $m$, but also on the internal structure of the palindrome, we cannot obtain a lower bound on the expected number of palindromic factors just using standard balls-and-bins considerations. Instead, we use a more powerful technique. This technique is based on the asymptotic estimates of the number of words of length $n$ avoiding a given fixed factor.

\section{Lower bound through avoidance of factors}

Below we assume that a $k$-ary alphabet $\Sigma$ is fixed, $k\ge 2$, all words are over $\Sigma$, and $\pal$ is the set of all palindromes over $\Sigma$. We say that a word $u$ \emph{avoids} a word $w$ if $w$ is not a factor of $u$. Let $A_w(n)$ be the number of words of length $n$ avoiding the word $w$ and let $\E(n,k,m)$ be the expected number of distinct palindromes of length $m$ in the words of length $n$. 

\begin{lemma}
\begin{equation} \label{Enkm}
\E(n,k,m)=\sum_{\substack{|w|=m,\\w\in\pal}} \Big(1-\frac{A_w(n)}{k^n}\Big).
\end{equation}
\end{lemma}

\begin{proof}
Consider the function on words that equals 1 if a word contains a given length $m$ palindrome $w$ and 0 otherwise. Applied to a random word, this function becomes a random variable with the expectation $\big(1-\frac{A_w(n)}{k^n}\big)$. This expectation is exactly the probability for a random word of length $n$ to contain $w$. Clearly, by the linearity of expectation, $\E(n,k,m)$ is the sum of such expectations over all palindromes of length $m$.
\end{proof}

To make use of \eqref{Enkm} for the estimation of $\E(n,k)=\sum_{m=1}^n \E(n,k,m)$, we have to estimate the number of words avoiding a given palindrome. For this purpose, we use the technique developed by Guibas and Odlyzko in \cite{GuOd78,GuOd81}. To formulate some of their results, we need to introduce some important notions. Recall that a word $u$ is a \emph{border} of a word $w$ if $u$ is both a prefix and a suffix\footnote{This definition deviates slightly from the usual one, which excludes the trivial case $u=w$.} of $w$. With each word $w$ of length $m$ we associate its \emph{border array}, which is a word $\hat{w}[1..m]$ over $\{0,1\}$ such that $w[i]=1$ if and only if $w$ has a border of length $m{-}i{+}1$. The border array can be interpreted as the array of coefficients of a real-valued polynomial $f_w(x)$ such that $\hat{w}[i]$ is the coefficient of $x^{m-i}$. We refer to this polynomial with 0-1 coefficients as the \emph{border polynomial} of $w$. Since $\hat{w}[1]=1$, this polynomial has degree $m{-}1$.

\begin{example}
The word $w=aabaabaa$ has non-empty borders $w$, $aabaa$, $aa$, and $a$. Its border array $\hat w$ equals $10010011$ and its border polynomial is $f_w(x)=x^7+x^4+x+1$.
\end{example}

\begin{theorem}[{\cite{GuOd78,GuOd81}}] \label{the:guod}
1) The number $A_w(n)$ of words of length $n$ avoiding a given word $w$ of length $m>3$ is
\begin{gather}
A_w(n)= C_w\theta_w^n+O(1.7^n), \label{eq:Awn}\\
 \text{ where }\theta_w= k-\frac 1{f_w(k)}-\frac{f'_w(k)}{f_w^3(k)}-O\Big(\frac{m^2}{k^{3m}}\Big),\quad
C_w= \frac{1}{1-(k-\theta)^2f'_w(\theta)}\,. \label{eq:thwCw}
\end{gather}
2) The condition $f_u(k)<f_w(k)$ implies $A_u(n)\le A_w(n)$ for all $n\ge 0$ and, in particular, $\theta_u\le \theta_w$.
\end{theorem}

\begin{lemma} \label{le1}
1) For words $u$ and $w$, one has $f_u(k)<f_w(k)$ if and only if $\hat{u}<\hat{w}$, where  $\hat{u}$ and $\hat{w}$ are treated as binary numbers.\\
2) For any $m$, $\max_{|w|=m} \theta_w=\theta_{a^m}$. 
\end{lemma}

\begin{proof}
1) The comparison of $\hat{u}$ and $\hat{w}$ as binary numbers has the same result as the comparison of them as $k$-ary numbers; but the number having $\hat{w}$ as its $k$-ary notation is exactly $f_w(k)$ by the definition of $f_w(x)$.

2) The border array of $a^m$ equals $1^m$ and thus represents the biggest number that can be written in binary in $m$ bits. Now the statement follows from statement 1 and  Theorem~\ref{the:guod}(2). 
\end{proof}


Applying Lemma~\ref{le1} and Theorem~\ref{the:guod}(2),  we see that
\begin{equation} \label{Awn}
A_{w}(n)\le A_{a^m}(n)\quad \text{ for any palindrome } w \text{ of length } m.
\end{equation}
Thus we can get the lower bound on the expected number of palindromic factors replacing $w$ in \eqref{Enkm} with the word $v=a^m$.
We have $f_v(x)=x^{m-1}+x^{m-2}+\cdots+x+1=(x^m-1)/(x-1)$, as we can assume $x>1$ since $k,\theta_v>1$. Hence,
\begin{equation}
f'_v(x)=\Big(\frac{x^m-1}{x-1}\Big)'=\frac{mx^{m-1}(x-1)-x^m+1}{(x-1)^2}=\frac{(m-1)x^m-mx^{m-1}+1}{(x-1)^2}\,.
\end{equation}
Substituting these formulas into \eqref{eq:thwCw} and performing straightforward transformations, we get
\begin{multline} \label{thetaam}
\theta_v=k-\frac{k-1}{k^m-1}-\frac{(k-1)\big((m-1)k^m-mk^{m-1}+1\big)}{(k^m-1)^3}+O\Big(\frac{m^2}{k^{3m}}\Big)\\
=k-\Big[\frac{k-1}{k^m}+\frac{k-1}{k^{2m}}+O\Big(\frac 1{k^{3m-1}}\Big)\Big]-\frac{(k^m+3)(k-1)\big((m-1)k^m-mk^{m-1}+1\big)}{k^{4m}}+O\Big(\frac{m^2}{k^{3m}}\Big)\\
=k-\frac{k-1}{k^m}-\frac{k-1}{k^{2m}}-\frac{m-1}{k^{2m-1}}+\frac{2m-1}{k^{2m}}-\frac{m}{k^{2m+1}}+O\Big(\frac{km+m^2}{k^{3m}}\Big)\\
=k-\frac{k-1}{k^m}-\frac{m(k-1)^2}{k^{2m+1}}+O\Big(\frac{km+m^2}{k^{3m}}\Big)\,,
\end{multline}
\begin{multline} \label{Cam}
C_v= \frac{1}{1-(k-\theta_v)^2f'_v(\theta)}=1+(k-\theta_v)^2f'_v(\theta)+O\big((k-\theta_v)^4{f'_v}^2(\theta)\big)\\
= 1+ \Big(\frac{k-1}{k^m}+O\Big(\frac{m}{k^{2m-1}}\Big)\Big)^2\Big( \frac{(m-1)\theta_v^m-m\theta_v^{m-1}+1}{(\theta_v-1)^2} \Big) +O\Big(\frac{m^2}{k^{2m}}\Big)
= 1+ O\Big(\frac{m}{k^{m}}\Big)
\end{multline}
Now we use \eqref{eq:Awn} to estimate the sum in \eqref{Enkm}. 

Since our goal is to estimate the ratio $\frac{\E(n,k)}{\sqrt{n}}$, we do not need to cope with arbitrary $m$. Namely, we put 
\begin{equation} \label{mpe}
m= 2(p_e+\varepsilon)=2(p_o+\varepsilon)+1, \text{ where }\varepsilon=O(1).
\end{equation}
Thus, $m=\log n+O(1)$. This is sufficient for reaching the declared goal because of the following

\begin{remark} \label{lognplusC}
If $m-\log n=g(n)$ for any growing function $g$, then $\E(m,k,n)=o(\sqrt{n})$, and then $\sum_{m=\log(n)+g(n)}^n \E(m,k,n)=o(\sqrt{n})$ (see Fig.~\ref{graphs}); the same observation is true for the symmetric case $m-\log n=-g(n)$.
\end{remark}

From \eqref{mpe} and \eqref{pe} we get $k^m=n\cdot k^{2\varepsilon}\cdot(1- O(\frac{\log n}{n}))$, $C_v=1+O(\frac{\log n}{n})$, and
\begin{equation} \label{thetaamk}
\frac{\theta_v}{k}=1-\frac{(k-1)(1+O(\frac{\log n}{n}))}{n\cdot k^{1+2\varepsilon}}+O\Big(\frac{\log n}{n^2}\Big)=
1-\frac{k-1}{n\cdot k^{1+2\varepsilon}}+O\Big(\frac{\log n}{n^2}\Big)\,.
\end{equation}
Substituting $(1-\alpha/n)^n=e^{-\alpha}(1+O(\alpha/n))$ for big $n$, we have
\begin{multline} \label{termam}
1-C_v\Big(\frac{\theta_v}{k}\Big)^n= 1-\Big(1+O\Big(\frac{\log n}{n}\Big)\Big)\Big(1-\frac{k-1}{n\cdot k^{1+2\varepsilon}}+O\Big(\frac{\log n}{n^2}\Big)\Big)^n\\
= 1-\Big(1+O\Big(\frac{\log n}{n}\Big)\Big)e^{-\frac{k-1}{k^{1+2\varepsilon}}+O(\frac{\log n}{n})}\Big(1+O\Big(\frac 1n\Big)\Big)=1-e^{-\frac{k-1}{k^{1+2\varepsilon}}}+O\Big(\frac{\log n}{n}\Big)
\end{multline}
Finally, from \eqref{Enkm} we obtain
\begin{equation} \label{sumam}
\E(n,k,m)\ge \Pal(k,m) \cdot\Big(1-C_v\Big(\frac{\theta_v}{k}\Big)^n\Big)=
\begin{cases}
k^\varepsilon\Big(1-e^{-\frac{k-1}{k^{1+2\varepsilon}}}\Big)\sqrt{n}+O\Big(\frac{\log n}{\sqrt{n}}\Big), & m\text{ is even,}\\
k^{\varepsilon}\Big(1-e^{-\frac{k-1}{k^{1+2\varepsilon}}}\Big)\sqrt{kn}+O\Big(\frac{\log n}{\sqrt{n}}\Big), & m\text{ is odd.}
\end{cases}
\end{equation}
In particular, we proved the lower bound of order $\sqrt{n}$ for $\E(n,k)$, finishing the proof of Theorem~\ref{main}(1). Furthermore, consider the function $g(k,\varepsilon)=k^\varepsilon\big(1-e^{-\frac{k-1}{k^{1+2\varepsilon}}}\big)$. Clearly, $g(k,0)=\Omega(1)$. For odd $m$, $\varepsilon=0$ means that $p_o$ is integer. By the definition of $p_o$, for $p_o=i$ we have $n=n_i=k^{2i+1}+2i$. So if we take the sequence  $\{n_i\}_1^\infty$ and $m=2i+1$, we obtain $\frac{\E(n_i,k,m)}{\sqrt{n}}=\Omega(\sqrt{k})$. Comparing this to Lemma~\ref{upper}, we obtain statement 4 of Theorem~\ref{main}. On the other hand, let us show that $g(k,\varepsilon)=\Omega(k^{-|\varepsilon|})$ for any $\varepsilon$. Indeed, if $\varepsilon>0$, then the Maclaurin series for $e^{-\frac{k-1}{k^{1+2\varepsilon}}}$ is alternating and monotonely decreasing in absolute value, which gives us $g(k,\varepsilon)=k^{-\varepsilon}(1+o(1))$. If $\varepsilon<0$, then
$$
g(k,\varepsilon)=k^{-|\varepsilon|}\Big(1-\big(e^{-\frac{k-1}{k}}\big)^{k^{2|\varepsilon|}}\Big)>k^{-|\varepsilon|}\Big(1-e^{-\frac12}\Big)=\Omega(k^{-|\varepsilon|})\,.
$$
For any $n$ and the odd number $m=2(p_o+\varepsilon)+1$ which is the closest odd integer to $2p_o+1$, the absolute value of $\varepsilon$ is at most $1/2$. Then for this $m$ we have $\frac{\E(n_i,k,m)}{\sqrt{n}}=\Omega(1)$. According to Remark~\ref{antiphase}, there is a sequence $\{n_i\}_1^\infty$ (more precisely, one can take $n_i=k^{2i}+2i-1$) such that $\frac{\E(n_i,k)}{\sqrt{n}}=O(1)$. Thus, we finished the proof of Theorem~\ref{main}(3).

\medskip
Note that the statement 2 of Theorem~\ref{main} is not proved yet: from statements 3 and 4 it follows that the limit doest not exist for $k$ big enough, while we have to prove this fact for all $k$. To do this, we need to tighten both upper and lower bounds.

\section{Tight two-sided bounds}

\begin{lemma} \label{whp}
With high probability, all borders of a randomly chosen palindrome of length $m$ have lengths less than $\lfloor\log m\rfloor$. 
\end{lemma}

\begin{proof}
By the definition of a border, any border of a palindrome is a palindrome. Thus, a palindrome has a border of a given length if and only if it begins with a palindrome of this length. A random word of length $2c$ or $2c{+}1$ is a palindrome with probability $k^{-c}$. Hence, by the union bound, the probability for a random word to begin with a palindrome of length at least $2c$ is less then 
$$
2\cdot\sum_{i=c}^{\infty} k^{-c}=\frac{2k}{k-1}\cdot k^{-c}\,.
$$
If we take $c=\lfloor\frac {\log m}2\rfloor$, this probability will be $O(m^{-1/2})$. Thus, a palindrome of length $m$ has no borders of length at least $2\cdot\lfloor\frac{\log m}2\rfloor\le \lfloor\log m\rfloor$ with probability $1-O(m^{-1/2})$.
\end{proof}

Now pick a palindrome $w$ of length $m$ at random. By Lemma~\ref{whp}, its border array $\hat w$ looks like $10\cdots0u$, where $|u|\le\lfloor\log m\rfloor$ with high probability. Since $w$ definitely has a one-letter border, $|u|>0$. Therefore, Theorem~\ref{the:guod}(2) and Lemma~\ref{le1} allow us to take $x^{m-1}+1$ and $x^m+x^{\lfloor\log m\rfloor}$ as the lower and the upper bound for $f_w(x)$ when estimating $A_w(n)$ (the lower bound works always and the upper bound works with high probability).

Now we take the function $x^{m-1}+x^c$, where the number $c\in\{0,1,\ldots,\lfloor\log m\rfloor\}$ is unspecified, as $f_w$, and compute $A_w(n)$ from it. We have
$f'_w(x)=(m-1)x^{m-2}+cx^{c-1}$. Similar to \eqref{thetaam} and \eqref{Cam} we obtain
\begin{multline} \label{thetaaba}
\theta_w=k-\frac{1}{k^{m-1}+k^c}-\frac{(m-1)k^{m-2}+ck^{c-1}}{(k^{m-1}+k^c)^3}+O\Big(\frac{m^2}{k^{3m}}\Big)\\
=k-\frac{1}{k^{m-1}}+\frac{1}{k^{2m-2-c}}-\frac{m-1}{k^{2m-1}}+O\Big(\frac{k^{2c+3}+k^{c+2}m+m^2}{k^{3m}}\Big)\,,
\end{multline}
\begin{multline} \label{Caba}
C_w= \frac{1}{1-(k{-}\theta_w)^2f'_w(\theta)}
= 1+ \Big(\frac{1}{k^{m-1}}+O\Big(\frac{m-k^{c+1}}{k^{2m-1}}\Big)\Big)^2\Big((m-1)\theta_w^{m-2}+c\theta_w^{c-1}\Big)+O\Big(\frac{m^2}{k^{2m}}\Big)\\
= 1+ O\Big(\frac{m}{k^m}\Big)
\end{multline}
Next we substitute $m= 2(p_e+\varepsilon)=2(p_o+\varepsilon)+1$, where $\varepsilon=O(1)$. Recalling that $k^c=O(m)$, we obtain, similar to \eqref{thetaamk}, \eqref{termam},
\begin{gather} 
\frac{\theta_w}{k}=1-\frac{1}{n\cdot k^{2\varepsilon}}+O\Big(\frac{\log n}{n^2}\Big)\,,\label{thetaabak}\\
1-C_w\Big(\frac{\theta_w}{k}\Big)^n = 1-e^{-\frac{1}{k^{2\varepsilon}}}+O\Big(\frac{\log n}{n}\Big)\,. \label{termaba} 
\end{gather}
The resulting asymptotic formulas are independent of $c$. So \eqref{termaba} gives the asymptotic value of a term in \eqref{Enkm} with high probability. All terms falling into the remaining small group can be bounded using \eqref{termam}, which gives a formula equivalent to \eqref{termaba} up to a multiplicative constant. Hence we can substituite \eqref{termaba} for \emph{all} terms in \eqref{Enkm}, getting finally
\begin{equation} \label{sumaba}
\E(n,k,m)= \Pal(k,m) \cdot\Big(1-C_w\Big(\frac{\theta_w}{k}\Big)^n\Big)=
\begin{cases}
k^\varepsilon\Big(1-e^{-\frac{1}{k^{2\varepsilon}}}\Big)\sqrt{n}+O\Big(\frac{\log n}{\sqrt{n}}\Big), & m\text{ is even,}\\
k^{\varepsilon}\Big(1-e^{-\frac{1}{k^{2\varepsilon}}}\Big)\sqrt{kn}+O\Big(\frac{\log n}{\sqrt{n}}\Big), & m\text{ is odd.}
\end{cases} 
\end{equation}
To extract the bounds on $\frac{\E(n,k)}{\sqrt{n}}$ from \eqref{sumaba}, we look at the function appeared as the coefficient of $\sqrt{n}$. 

\begin{remark} \label{xex2}
The function $f(x)=x(1-e^{-1/x^2})$ behaves over the interval $(0,\infty)$ as follows:
\begin{enumerate}
\item $f(x)\sim 1/x$ (up to a cubically small term) as $x\to\infty$; more precisely, for $x>1$ one has $f(x)=\frac 1x-\frac 1{2x^3}+\frac 1{6x^5}-\Delta$, where $0<\Delta<\frac 1{24x^7}$;
\item $f(x)\sim x$ (up to an exponentially small term $-xe^{-1/x^2}$) as $x\to0$;
\item $f(x)$ has a single maximum $\chi\approx 0.6382$ at the point $x_0\approx 0.8921$ and is nearly constant around this point (e.g., $f(1)=1-1/e\approx 0.6321$).
\end{enumerate}
\end{remark}

Now consider $F(k,\varepsilon)=\sum_{i=-\infty}^\infty f(k^{\varepsilon+i})$. By Remark~\ref{xex2}, this series clearly converges, being bounded by the sum of two geometric series with the same denominator $k^{-\varepsilon}$. Furthermore, $F(k,\varepsilon)$ is periodic with the period 1 for any fixed $k\in{\mathbb N}\backslash\{1\}$. 

To make the computation of the sum $\E(n,k)=\sum_{m=1}^n \E(n,k,m)$ easier, we first discard most of its terms, leaving $\sum_{m=\lfloor\log n\rfloor-c}^{\lfloor\log n\rfloor+c} \E(n,k,m)$, for some constant $c$. This produces an error of order $k^{-c/2}\sqrt{n}$ (see Fig.~\ref{graphs}; cf. Remark~\ref{lognplusC}). Every term of the remaining sum can be computed by the formula \eqref{sumaba}. Next we replace this finite sum with an infinite sum of terms \eqref{sumaba}, taken for all $\varepsilon$ such that $-\infty<\varepsilon<\infty$ and either $p_e+\varepsilon$ or $p_o+\varepsilon$ is an integer. By Remark~\ref{xex2}, the sum we thus added is also of order $k^{-c/2}\sqrt{n}$. Hence, we totally change $\E(n,k)$ by an amount of order $k^{-c/2}\sqrt{n}$. Since the constant $c$ can be taken big enough, we can neglect this change in our considerations and identify $\E(n,k)$ with this infinite sum, getting
\begin{equation} \label{EnkF}
\E(n,k)\approx \Big(F\big(k,\varepsilon\big)\sqrt{k}+F\big(k,\varepsilon+\frac 12\big)\Big)\sqrt{n},\text{ where } p_o(n,k)+\varepsilon\in\mathbb Z\,.
\end{equation}
In order to prove Theorem~\ref{main}(2), it remains to show that the function $F(k,\varepsilon)$ has no period $1/2$ for any fixed $k\in{\mathbb N}\backslash\{1\}$. For this, let us first consider $F(k,0)$ and $F(k,1/2)$. From \eqref{sumaba} and Remark~\ref{xex2} we have
\begin{equation}
F(k,0)= 1-\frac 1e+\frac 2{k-1}-\frac 1{2(k^3-1)}+\frac 1{6(k^5-1)}-\frac 1{ke^{k^2}}-\Delta,\text{ where }\Delta<\frac 1{24(k^7-1)},
\end{equation}
yielding $F(k,0)\ge 1 -\frac 1e +\frac 2{k-1}-\frac 1{2(k^3-1)}$ for $k\ge 3$. Similarly,
\begin{equation}
F(k,1/2)\le \frac {2\sqrt{k}}{k-1}-\frac {k^{3/2}}{2(k^3-1)}+\frac {k^{5/2}}{6(k^5-1)}-\frac 1{\sqrt{k}e^k}\,.
\end{equation}
Then
\begin{equation} \label{fk0}
F(k,0)- F(k,1/2)\ge 1 -\frac 1e+\frac {2(1-\sqrt{k})}{k-1}+\frac {(k^{3/2}-1)}{2(k^3-1)}-\frac {k^{5/2}}{6(k^5-1)}+\frac 1{\sqrt{k}e^k}\,.
\end{equation}
The difference \eqref{fk0} can be checked by hand or by computer-assisted symbolic computation to be positive for any $k\ge4$. Hence, the function $F(k,\varepsilon)$ has no period $1/2$ in these cases. This implies that no limit $\lim_{n\to\infty} \frac{\E(n,k)}{\sqrt{n}}$ exists according to \eqref{EnkF}. The cases $k=2$ and $k=3$ require a separate analysis, but since $k$ is fixed, this is feasible.  It appears that in each case $F(k,\varepsilon)$ has a single maximum and a single minimum on any interval of length $1$, and thus has no period 1/2. More detailed, $\max F(2,\varepsilon)\approx 2.55775$ at the point $x_0\approx 0.398$ and $\min F(2,\varepsilon)\approx 2.55647$ at the point $x_0\approx -0.103$; $\max F(3,\varepsilon)\approx 1.62212$ at the point $x_0\approx -0.251$ and $\min F(3,\varepsilon)\approx 1.60452$ at the point $x_0\approx 0.255$. This finally proves statement 2 and then Theorem~\ref{main}.

\begin{remark}
The difference between the maximum and the minimum in the binary case is really tiny; to prove its existence, all terms given in Remark~\ref{xex2}(1,2) are essential.
\end{remark}

With all the bounds obtained, the following proposition is easy.

\begin{proposition}
(1) $\lim_{k\to\infty} \underline{C}(k)=3-1/e$.\\
(2) $\lim_{k\to\infty}\overline{C}(k)/\sqrt{k}=\chi$, where $\chi\approx 0.6382$ is the maximum of the function $f(x)=x(1-e^{-1/x^2})$ in the interval $(0,\infty)$.
\end{proposition}

\begin{proof}
For statement 1, note that \eqref{sumaba} gives us a coefficient of order $k^{1/2-|\varepsilon|}$ for the number of odd-length palindromes and a coefficient of order $k^{-|\varepsilon|}$ for the number of even-length palindromes. So we can get a coefficient of order $O(1)$ only by taking a subsequence of $n$'s such that the corresponding $\varepsilon$'s tend to $1/2$. In this case, even palindromes contribute $1-1/e+O(1/k)$ and odd-length palindromes contribute $2+O(1/k)$, whence the result.

Let us turn to statement 2. Let $\varepsilon_0=\log x_0$, where $x_0$ is defined in Remark~\ref{xex2}(3). One can choose a subsequence of $n$'s such that the corresponding sequence of $\varepsilon$'s converges to $\varepsilon_0$. Then the expectations $\E(n,k,m)$, corresponding to these $n$'s and $\varepsilon$'s, form a sequence, equivalent to $\chi\sqrt{kn}$ as $n\to\infty$, see \eqref{sumaba}. On the other hand, the function $\chi\sqrt{kn}$ bounds any sequence of expectations $\E(n,k,m)$ from above. It remains to note that at most one term $\E(n,k,m)$ for a given $n$ is proportional to $\sqrt{kn}$ while all others are proportional to $k^c\sqrt{n}$ for some $c\le0$. The result now follows.
\end{proof}

\section{Numerical results and possible extensions}

Below we give, in Table~\ref{tab}, the numerical estimates for some particular values of $\underline{C}(k)$ and $\overline{C}(k)$ together with the corresponding values of $\varepsilon$ such that $p_o+\varepsilon$ is an integer and $|\varepsilon|\le 1/2$. We compare these numerical values against the experimental data on the palindromic richness of random words. The problem of counting distinct palindromic factors in a word can be efficiently solved: see \cite{GPR10} for an offline algorithm and \cite{KRS13} for an online one. This makes possible the experiments with long random words. For each length, Table~\ref{tab} contains the average number of palindromes for 1000 experiments, divided by $\sqrt{n}$. The experimental data agree quite well with the theory; for longer words the agreement is better. We also mention a special situation with the binary alphabet: the difference $\overline{C}(2)-\underline{C}(2)$ is very small, and the values of $\underline{\varepsilon}$ and $\overline{\varepsilon}$ are ``swapped'' compared to bigger alphabets.

\begin{table} 
\caption{Theoretical values of the constants $\underline{C}(k)=\liminf_{n\to\infty}\frac{\E(n,k)}{\sqrt{n}}$ and $\underline{C}(k)=\limsup_{n\to\infty}\frac{\E(n,k)}{\sqrt{n}}$, the corresponding values of the distance $\varepsilon$ between $p_o(n,k)$ and the closest integer, and the experimental data on the number of distinct palindromes in random words of lengths fitting to the obtained values of $\varepsilon$.}
\centerline{
\tabcolsep=4pt
\begin{tabular}{|r|l|r|l|r|r|l|r|l|}
\hline
$k$&$\underline{C}(k)$&$\underline{\varepsilon}$&$\overline{C}(k)$&$\overline{\varepsilon}$&$\underline{n}$&$Pals_{\underline{n}}/\sqrt{\underline{n}}$ &$\overline{n}$&$Pals_{\overline{n}}/\sqrt{\overline{n}}$\\
\hline
2&6.17315& -0.103&6.17368&0.398&618843800&6.17171&1238545800&6.17276\\
3&4.40121& 0.255&4.41410&-0.251&8188445&4.40052&24940577&4.41358\\
4&3.81315& 0.360&3.85763&-0.167&24747862&3.81195&6657745&3.85465 \\
5&3.51925& 0.409 &3.60893&-0.129&13076560&3.51834&2914038&3.60581\\
6&3.34259& 0.438 &3.48553&-0.108&2096750&3.34202&14840282&3.48520\\
10&3.02693& 0.485 &3.41133&-0.071&1071524&3.02544&13842043&3.41175 \\
50&2.70152& -0.485 &5.09183&-0.032&5877686&2.70007&160063&5.08441 \\
\hline
\end{tabular}
}
\label{tab}
\end{table}

\medskip
Finally, we point out that the technique used in this paper can be applied to computing the expected numbers of other types of repetitions in random words. For example, it is quite easy to show that the expected number of squares in a $k$-ary word of length $n$ is $\sqrt{n}$; moreover, the ratio of this number and $\sqrt{n}$ tends to a constant as $k\to \infty$. Indeed, squares are very much alike the even-length palindromes (e.g., the left graph of Fig.~\ref{graphs} suits for squares as well), and there is no analog of odd-length palindromes to disturb the general picture. The only significant difference between squares and even palindromes is in their borders: palindromes usually have only short borders, while a square of length $n$ always has the border of length $n/2$, and \emph{with high probability has no longer borders}. The corresponding difference in border polynomials affects the constant before the $\sqrt{n}$ term, but not the term itself (compare \eqref{sumam} against \eqref{sumaba}). Thus, the analog of \eqref{sumaba} can be obtained, with slightly different constant and without the alternative for odd-length palindromes.

\bibliographystyle{fundam}          
\bibliography{my_bib}            

\end{document}